\documentclass{amsart}
\usepackage[utf8]{inputenc}
\usepackage[T1]{fontenc}
\usepackage{mathtools}
\usepackage[backend=biber,
            style=alphabetic,
            url=false
            ]{biblatex}
\usepackage{amsmath}
\usepackage{amsfonts}
\usepackage{amssymb}
\usepackage{thmtools}
\usepackage{setspace}
\usepackage{amsthm}
\usepackage{todonotes}
\usepackage{hyperref}
\usepackage{lmodern}
\usepackage{tikz-cd}
\usepackage{verbatim}
\usepackage[capitalise]{cleveref}
\usepackage{enumitem}
\usepackage{datetime}
\usepackage{bm}
\usepackage{comment}
\crefname{enumi}{}{}
\Crefname{enumi}{}{}
\newif\iflong
\longtrue
\usetikzlibrary{decorations.markings}
\tikzset{negated/.style={
        decoration={markings,
            mark= at position 0.5 with {
                \node[transform shape] (tempnode) {$\backslash$};
            }
        },
        postaction={decorate}
    }
}
\theoremstyle{plain}
\declaretheorem[name={Theorem}]{theorem}
\declaretheorem[name={Lemma}, sibling=theorem]{lemma}

\declaretheorem[name={Proposition}, sibling=theorem]{proposition}
\declaretheorem[name={Claim}, numberwithin=theorem]{claim}
\declaretheorem[name={Corollary}, sibling=theorem]{corollary}

\declaretheorem[name={Example}, sibling=theorem]{example}
\theoremstyle{definition}
\declaretheorem[name={Definition}]{definition}

\declaretheorem[name={Remark},numbered=no]{remark}

\DeclareMathOperator{\enters}{\searrow}
\newcommand{\A}{\mathcal{A}}

\newcommand{\B}{\mathcal{B}}
\newcommand{\C}{\mathcal{C}}

\newcommand{\LR}{\Leftrightarrow}
\newcommand{\ra}{\rightarrow}
\newcommand{\Ra}{\Rightarrow}
\newcommand{\La}{\Leftarrow}

\newcommand{\mc}[1]{\mathcal{#1}}
\newcommand{\mrm}[1]{\mathrm{#1}}

\newcommand{\ol}[1]{\overline{#1}}

\renewcommand{\phi}{\varphi}
\newcommand{\Sinf}[1]{\Sigma_{#1}}
\newcommand{\dSinf}[1]{d\text{-}\Sigma_{#1}}

\newcommand{\Pinf}[1]{\Pi_{#1}}
\newcommand{\Sicom}[1]{\Sigma^{\mathrm{c}}_{#1}}
\newcommand{\dSicom}[1]{d\text{-}\Sigma^{\mathrm{c}}_{#1}}
\newcommand{\Picom}[1]{\Pi^{\mathrm{c}}_{#1}}

\newcommand{\dSigma}{d\text{-}\Sigma}

\setlist*[enumerate,1]{
 label=(\arabic*),
}
\newlist{inlinelist}{enumerate*}{1}
\setlist*[inlinelist,1]{
label=(\arabic*),
}

\begin{document}
\title{The complexity of Scott sentences of scattered linear orders 
}

\author[R. Alvir]{Rachael Alvir}
\author[D. Rossegger]{Dino Rossegger}

\address{Department of Mathematics, University of Notre Dame
}
\email{ralvir@nd.edu}

\address{Department of Pure Mathematics, University of Waterloo
}
\email{drossegg@uwaterloo.ca}

\thanks{
The first author was supported by NSF grant DMS-1547292. The second author was supported by the Austrian Science Fund
FWF through project P 27527. We are grateful to Julia Knight for many helpful discussions and comments.}

\begin{abstract}
  Given a countable scattered linear order $L$ of Hausdorff rank $\alpha < \omega_1$ we show that it has a $\dSinf{2\alpha+1}$ Scott sentence. Ash~\cite{ash1986} calculated the back and forth relations for all countable well-orders. From this result we obtain that this upper bound is tight, i.e.,  for every $\alpha < \omega_1$ there is a linear order whose optimal Scott sentence has this complexity. We further show that for all countable $\alpha$ the class of Hausdorff rank $\alpha$ linear orders is $\pmb \Sigma_{2\alpha+2}$ complete.
\end{abstract}
\maketitle
\section{Introduction}
Scott~\cite{scott1963} showed that every countable structure $\mc A$ can be described up to isomorphism among countable structures by a single sentence of $L_{\omega_1\omega}$, called the \emph{Scott sentence} of $\A$. The logic $L_{\omega_1\omega}$ extends finitary first-order logic by allowing countable disjunctions and conjunctions; if the conjunctions and disjunctions are over c.e. sets of formulas, the sentence is called \emph{computable}.

Although there is no prenex normal form for formulas of $L_{\omega_1 \omega}$, there is a normal form which allows every $L_{\omega_1\omega}$ formula to be measured by a kind of quantifier complexity. Every $L_{\omega_1\omega}$ formula is logically equivalent to a $\Sinf{\alpha}$ or $\Pinf{\alpha}$ infinitary formula for some countable ordinal $\alpha$. A formula that is the conjunction of a $\Sinf{\alpha}$ and a $\Pinf{\alpha}$ formula is called a $\dSinf{\alpha}$ formula.

Closely related to the complexity of a structure's Scott sentence is its Scott rank. Scott rank is a well-studied notion in computable structure theory and descriptive set theory. For example, in~\cite{ash2000} Scott ranks for classes of computable structures such as ordinals, vector spaces, and superatomic Boolean algebras are calculated. However, there exist several incompatible but closely related definitions of Scott rank in the literature, see~\cite{ash2000} for a discussion.  Montalb\'an~\cite{montalban2015} attempted to standardize Scott rank by proposing that a structure $\A$'s \emph{categoricity Scott rank}, the least $\alpha$ such that $\A$ has a $\Pinf{\alpha+1}$ Scott sentence, is the most robust such notion. However, in~\cite{alvir2018} it is shown that a least quantifier-complexity Scott sentence for a structure exists, from which one can calculate a structure's categoricity Scott rank as well as the other notions of Scott rank discussed in~\cite{montalban2015}. Therefore, having a least complexity Scott sentence gives a more finegrained picture than the Scott rank of a structure.

In the case that the structure is computable, the least complexity computable Scott sentence (if it exists) gives an upper bound on the complexity of the set of indices of its computable copies.
The complexity of a computable structure's index set has seen a lot of interest in the last few years. Although the least complexity of a structure's Scott sentence does not always establish the complexity of its index set~\cite{knight2014}, the two have been closely related in practice. Often, as in~\cite{calvert2006b}, index set results for several classes of algebraic structures were first conjectured by finding an optimal Scott sentence.

In this article we investigate the complexity of Scott sentences of scattered linear orders. The most prominent examples of scattered linear orders are well-orders. Ash~\cite{ash1986} calculated their back and forth relations. His results imply that for any countable $\alpha$, $\omega^{\alpha}$ has a $\Pi_{2 \alpha + 1}$ optimal Scott sentence. Scattered linear orders have an inductive characterization due to Hausdorff~\cite{hausdorff1908}. The Hausdorff rank of a linear order is the least number of induction steps necessary to obtain it. From McCoy's work \cite{mccoy2003} on $\Delta_2^0$-categoricity of computable linear orders, one can see that every Hausdorff rank 1 linear order has a $\Sigma_4$ Scott sentence. Nadel~\cite{nadel_scott_1974} gave a first upper bound on the complexity of linear orders of countable Hausdorff rank $\alpha$ by showing that the complexity of the Scott sentence of an order of Hausdorff rank $\alpha$ is less than $\Pi_{\omega\cdot(\alpha+2)}$. 

In this article we obtain much better bounds than those provided by Nadel~\cite{nadel_scott_1974}.
We show that all scattered linear orders of Hausdorff rank $\alpha$ have a $d$-$\Sigma_{2 \alpha + 1}$ Scott sentence, and that this bound is tight in the sense that there is a scattered linear order of Hausdorff rank $\alpha$ for which this sentence is optimal. For the Hausdorff rank $1$ case, we classify the linear orders which have $\Pi_3$ Scott sentences, show that their Scott sentences are optimal and prove that all other Hausdorff rank $1$ linear orders have $\dSinf{3}$ optimal Scott sentences. Results by Frolov and Zubkov (unpublished) show that for finite $\alpha>1$ there are orders with a Scott sentence of complexity $\Pi_n$ for any $n$, $3<n<2\alpha$. Since for every $\alpha$, there exists a linear order of Hausdorff rank $\alpha$ having a $\dSinf{2\alpha+1}$ Scott sentence, one cannot obtain optimal bounds in general.

Nadel~\cite{nadel_scott_1974} observed that the Scott rank of any computable structure must be less than or equal to $\omega_1^{\mrm{CK}}+1$ where $\omega_1^{\mrm CK}$ is the first non-computable ordinal. A computable structure with categoricity Scott rank $\omega_1^{\mrm{CK}}$ or $\omega_1^{\mrm{CK}}+1$ is said to have \emph{high Scott rank}. Note that if a structure has high Scott rank for one notion of Scott rank then this holds for every notion that appears in the literature.
Nadel~\cite{nadel_scott_1974} also showed that the Scott rank of a computable scattered linear order cannot be high. As a corollary of our results we get a new proof of this theorem. Several examples of structures with high Scott rank have been found. One of the first is due to Harrison~\cite{harrison_recursive_1968} who constructed a computable linear order of order type $\omega^{\mrm{CK}}_1\cdot(1+\eta)$. It has Scott rank $\omega_1^{\mrm{CK}}+1$. Calvert, Knight, and Miller~\cite{calvert_computable_2006} and Knight and Miller~\cite{knight_computable_2010} found examples of computable structures with Scott rank $\omega^{\mathrm{CK}}_1$. Recently, Harrison-Trainor, Igusa, and Knight~\cite{harrison-trainor_new_2018} gave examples of structures of Scott rank $\omega_1^{CK}$ whose computable infinitary theory is not $\aleph_0$-categorical.

We also obtain a new result on the Borel complexity of classes of scattered linear orders. It is well known that the class of scattered linear orders is $\pmb\Pi_1^1$ complete; see for instance~\cite{kechris2012classical}. We refine this picture by proving that the class of linear orders of Hausdorff rank $\alpha$ is $\pmb\Sigma_{2\alpha+2}$ complete.

In the rest of this section we review the concepts needed in this article. In \cref{sec:upperbound} we calculate the complexity of the Scott sentences and in \cref{sec:tightness} we show that our bounds are tight and calculate the Borel complexity of the class of linear orders of Hausdorff rank $\alpha$.

\subsection{Infinitary logic, back and forth relations and Scott sentences}

Every formula of $L_{\omega_1 \omega}$ is equivalent to one that is $\Sinf{\alpha}$ or $\Pinf{\alpha}$. We define the classes of $\Sigma_{\alpha}, \Pi_{\alpha}$ formulas inductively as follows:
\begin{enumerate}
  \item A formula $\phi(\ol x)$ is $\Sigma_0$ iff it is $\Pi_0$ iff it is a finitary quantifier-free formula.
    \item A $\Sigma_{\alpha}$ formula is a formula of the form $\bigvee_{i \in \omega} \exists \bar{x} \phi_i(\bar{x})$ where each $\phi_i$ is $\Pi_{\beta}$ for $\beta < \alpha$.
    \item A $\Pi_{\alpha}$ formula is the negation of a $\Sigma_{\alpha}$ formula. Equivalently, such a formula is of the form $\bigwedge_{i \in \omega}  \forall \bar{x} \phi_i(\bar{x})$ where each $\phi_i$ is $\Sigma_{\beta}$ for $\beta < \alpha$.
\end{enumerate}
A formula is said to be $X$-computable if all of its disjunctions and conjunctions are over $X$-$c.e$ sets of formulas. Notice that every infinitary formula is $X$-computable for some $X$.
We write $\Sigma_{\alpha}^{X}$, $\Pinf{\alpha}^{X}$ or $\dSinf{\alpha}^X$ if we want to emphasize that a formula is $X$-computable. If a formula is computable, we write $\Sicom{\alpha}$, $\Picom{\alpha}$ or $\dSicom{\alpha}$.

Scott~\cite{scott1963} proved that any countable structure can be described up to isomorphism among countable structures by a sentence in $L_{\omega_1\omega}$ -- its \emph{Scott sentence.} Alvir, Knight, and McCoy~\cite{alvir2018}, effectivizing a result of Miller~\cite{miller1983}, proved the following. 

\begin{theorem}[\cite{alvir2018}]\label{thm:rachael}
Let $\mc A$ be a countable structure and $X\subseteq \omega$.
If $\mc A$ has both a $\Sigma_{\alpha}^{X}$ Scott sentence and a $\Pi_{\alpha}^{X}$ Scott sentence, then it has a $\dSigma_{\beta}^{X}$ Scott sentence for some $\beta < \alpha$.
\end{theorem}

This result implies that every countable structure has exactly one least complexity Scott sentence in the following partial order:
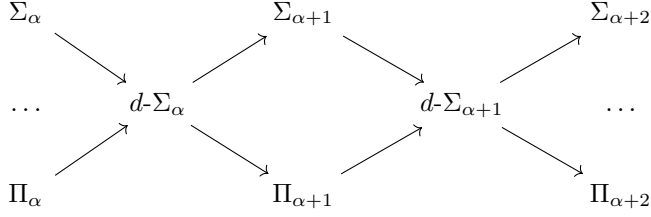
\vspace{1em}
\begin{center}
\begin{tikzcd}
\Sinf{\alpha}\ar[dr] & &\Sinf{\alpha+1}\ar[dr]& & \Sinf{\alpha+2}\\
\dots &\dSinf{\alpha}\ar[ur]\ar[dr]& & \dSinf{\alpha+1}\ar[dr]\ar[ur] & \dots\\
\Pinf{\alpha}\ar[ur] & &\Pinf{\alpha+1}\ar[ur]& & \Pinf{\alpha+2}
\end{tikzcd}
\end{center}
\vspace{1em}

Montalb\'an~\cite{montalban2015} proved that several conditions are equivalent to having a Scott sentence of a certain complexity.
\begin{theorem}[{\cite[Theorem 1.1.]{montalban2015}}]\label{thm:ssnopar}
  Let $\mc A$ be a countable structure and $\alpha<\omega_1$. Then the following are equivalent:
  \begin{enumerate}
    \item Every automorphism orbit is $\Sinf{\alpha}$ definable without parameters.
    \item $\mc A$ has a $\Pinf{\alpha+1}$ Scott sentence.
    \item $\mc A$ is uniformly boldface $\pmb{\Delta}_\alpha^0$-categorical.
    \item No tuple in $\A$ is $\alpha$-free.
  \end{enumerate}
\end{theorem}
The \emph{categoricity Scott rank} of $\A$ is the least $\alpha$ satisfying one of the above statements.
\begin{theorem}[{\cite[Theorem 2.5.]{montalban2015}}]\label{thm:sspar}
  Let $\mc A$ be a countable structure and $\alpha<\omega_1$. Then the following are equivalent:
  \begin{enumerate}
    \item $\mc A$ has a $\Sinf{\alpha+2}$ Scott sentence.
    \item There is a tuple $\ol a\in A^{<\omega}$ such that $(\mc A,\ol a)$ has a $\Pinf{\alpha+1}$ Scott sentence.
  \end{enumerate}
\end{theorem}
Let us briefly discuss the notions appearing in the above theorems. Central to the notion of $\alpha$-freeness are the $\leq_\alpha$ back and forth relations.
Given countable structures $\mc A$ and $\mc B$ in the same language, tuples $\ol a\in A^{< \omega} $ , $\ol b\in B^{< \omega}$ and an ordinal $\alpha$, we write  $(\A,\ol a)\leq_{\alpha} (\B,\ol b)$
if and only if every $\Pinf{\alpha}$ formula true of $\bar{a}$ in $\A$ is true of $\bar{b}$ in $\B$; equivalently, if every $\Sinf{\alpha}$ formula true of $\bar{b}$ in $\B$ is true of $\bar{a}$ in $\A$. One can also define the back and forth relations inductively as follows. For $\leq_1$ we use the same definition as above: $(\A,\ol a)\leq_1(\B, \ol b)$ if all $\Pinf{1}$ formulas true of $\ol a$ in $\A$ are true of $\ol b$ in $\B$. For $\alpha>1$,
\[
 (\A,\ol a)\leq_{\alpha} (\B,\ol b)\LR (\forall \beta<\alpha) (\forall \ol d \in B^{<\omega}) (\exists \ol c \in A^{<\omega}) \text{ such that } (\mc B,\ol b \ol d)\leq_\beta (\mc A,\ol a \ol c) .\]
If $\mc A=\mc B$ and the structure is clear from context we often abuse notation and write $\ol a\leq_{\alpha}\ol b$. For a thorough treatment of back and forth relations including the case where $\alpha=0$ see~\cite[Chapter 15]{ash2000}.

Let $\A$ be a structure. A tuple $\ol a \in A^{<\omega}$ is \emph{$\alpha$-free} if
\[(\forall \beta<\alpha)\ (\forall \ol b)\  (\exists\ol a',\ol b')\   \left(\ol a\ol b \leq_\beta \ol a'\ol b' \quad \text{ and }\quad  \ol a \not \leq_\alpha \ol a'\right).\]
Given $\mc A$ and $\ol c \in A^{<\omega}$ we say that a tuple $\ol a\in A^{<\omega}$ is \emph{$\alpha$-free over $\ol c$} if it is $\alpha$-free in the structure $(\mc A,\ol c)$.

Uniform (boldface) $\pmb{\Delta}_\alpha^0$ categoricity and uniform (lightface) $\Delta_\alpha^0$ categoricity are notions studied in computable structure theory. A structure $\mc A$ is \emph{uniformly (lightface) $\Delta_\alpha^0$ categorical} if there is a $\Delta_\alpha^0$ operator $\Gamma$ such that for any isomorphic copy $\mc B$, $\Gamma^{\mc A \oplus \mc B}$ is an isomorphism between $\mc A$ and $\mc B$. It is \emph{uniformly (boldface) $\pmb{\Delta}_\alpha^0$ categorical} if it is uniformly $\Delta_{\alpha}^0$ categorical relative to some fixed oracle $X\subset \omega$, i.e., $\Gamma^{\mc A\oplus \mc B \oplus X}$ is an isomorphism between $\mc A$ and $\mc B$ for any isomorphic copy $\mc B$. We will use this in \cref{sec:tightness}. 

\subsection{Index sets} Given a structure $\mc A$, its \emph{index set} $I_\A$ is the set of indices of its computable isomorphic copies, where we identify
 a structure with its atomic diagram. We can also look at the index set of a structure relative to a set $X$, the set $I_\A^X$ of indices of $X$-computable isomorphic copies of $\A$.
The complexity of the Scott sentence gives an upper bound on the complexity of the index set. More formally, if $\A$ has a $\Sigma^X_\alpha$ Scott sentence, then any $\Sigma_\alpha^0(X)$ complete set can decide membership of $I_\A^X$. On the other hand, if $I_\A^X$ is Turing-complete for $\Sigma_\alpha^0(X)$ sets, then $\A$ cannot have a Scott sentence simpler than $\Sigma^X_\alpha$.
Since every infinitary sentence is computable relative to some set $X$, it is sufficient to compute the complexity of $I^X_\A$ for all $X$ to show that a Scott sentence for a structure $\A$ is optimal.

\subsection{Linear orders}
A linear order is \emph{scattered} if it does not have a dense suborder. Hausdorff~\cite{hausdorff1908} inductively constructed classes of linear orders $HR_\alpha$ and showed that a countable linear order is scattered if and only if it is in $HR_\alpha$ for some countable $\alpha$. The \emph{Hausdorff rank} of a linear order $L$ is the least $\alpha$ such that $L\in HR_\alpha$. In the literature several different definitions of this hierarchy exist; we will use the following:
\begin{definition}\label{def:hr}\leavevmode
  \begin{enumerate}
    \item $HR_0:= \{ n:n<\omega\}\footnote{Here, $n$ refers to the order type. Note that we include the empty order, $0$.}$,
    \item for countable $\alpha>0$, $HR_\alpha$ is the least class of all linear orders of the form $\sum_{i\in \zeta} L_i$ for $L_i\in \bigcup\{H_\beta : \beta<\alpha\}$ that is closed under finite sum.
  \end{enumerate}
  Let $r(L)$ be the least $\alpha$ with $L\in HR_\alpha$. Then $r(L)$ is the \emph{Hausdorff rank} of $\alpha$.
\end{definition}
Let $L$ be a linear order and $\sim$ an equivalence relation on $L$. We write $[x]_\sim$ for the equivalence class of $x$ modulo $\sim$. If $\sim$ is given from the context we sometimes omit the subscript. This is not to be mistaken with interval notation which we will aso use. In particular $[x,y]$ is the closed interval starting at $x$ and ending at $y$. Open ($(x,y)$), and half-open intervals ($[x,y)$, $(x,y]$) are defined analogously. We use the symbols $-\infty$ and $\infty$ to denote the start and end of linear order. For example, $(-\infty, x)$ is the initial segment of $L$ up to $x$.
\begin{definition}\label{def:block}
  Let $L$ be a linear order and $x,y\in L$. Then let
  \begin{enumerate}
    \item $x\sim_0 y$ if $x=y$,
    \item $x \sim_{1} y$ if $[x,y]$ or $[y,x]$ is finite,
    \item for $\alpha=\beta+1$, $x \sim_{\alpha} y$ if in $L /{\sim_{\beta}}$,
    $[x]_{\sim_\beta} \sim_1 [y]_{\sim_\beta}$,
    \item for $\alpha$ limit, $x \sim_{\alpha} y $ if for some  $\beta < \alpha$, $x \sim_{\beta} y$.
  \end{enumerate}
  The relation $\sim_1$ is commonly known as the \emph{block relation.} More generally, an \emph{$\alpha$-block} of a linear order $\mc L$ is an equivalence class modulo $\sim_{\alpha}$. The order $\mc L$ is said to be written in \emph{$\beta$-block form} if $L = \Sigma_{i \in I} L_i$ where each $L_i$ is a $\beta$-block. Note that a $\beta$-block, considered as a substructure, always has Hausdorff rank less than or equal to $\beta$.
\end{definition}
The next lemma follows easily by induction.
\begin{proposition}\label{prop:blockdefn}
  For countable $\alpha$, the relation $\sim_{\alpha}$ is $\Sigma_{2\alpha}$ definable.
\end{proposition}

\cref{def:hr,def:block} play nicely with each other.
\begin{proposition}\label{thm:finmodblock}
  For every linear order $L$, $r(L)=\alpha$ iff $\alpha$ is the least such that $L/{\sim_\alpha}$ is finite.
\end{proposition}
\begin{proof}
  $(\Ra)$ Assume $r(L) = \alpha$. We first claim that for any $\beta < \alpha$, $L/{\sim_\beta}$ is infinite.  If not, then for some $\beta < \alpha$ one can write $L$ as a finite sum of $\beta$-blocks. Therefore $L$ will be a finite sum of structures of rank $\leq \beta$, and so $r(L) \leq \beta$
  
  We now show by induction that $L/{\sim_\alpha}$ is finite. Clearly if $r(L)=0$ then the theorem holds. Assume it holds for $\beta<\alpha$ and that $r(L) = \alpha$. We know that for some $n$, $L = \Sigma_{i = 1}^n L_i$, and that each $L_i$ is a sum of the form $\Sigma_{j \in \zeta} L_{i,j}$ where for each $i$ and $j$, $r(L_{i,j}) < \alpha$.
  If $L/{\sim_\alpha}$ is greater than $n$, then there will be some $k<n$ and at least two distinct equivalence classes $[x]_{\sim \alpha}$, $[y]_{\sim \alpha}$ such that $x,y \in L_k$. Then $x,y \in (L_{k,i}, L_{k,i+j})$ for some $i$ and $j$. But the Hausdorff rank of $(L_{k,i}, L_{k,i+j})$ is strictly less than $\alpha$ since it is a finite sum of structures having Hausdorff rank $\beta<\alpha$. Then by hypothesis $(L_{k,i},L_{k,i+j})/{\sim_\alpha}=1$ and thus $x \sim_{\alpha}y$, a contradiction.

  $(\La)$ Follows by contraposition and $(\Ra)$.
\end{proof}



\section{Upper bounds}\label{sec:upperbound}
The goal of this section is to prove the following theorem.
\begin{theorem}\label{thm:mainthm}
  Let $L$ be a linear order with $r(L)=\alpha$. Then $L$ has a $\dSinf{2\alpha+1}$ Scott sentence.
\end{theorem}
The theorem is proved by induction on the Hausdorff rank of $L$. The following two definitions are central to our proof.
\begin{definition}
A \emph{finite partition} of a linear order $L$ is a finite tuple $(a_0, \ldots, a_{n-1}) \in L$ such that $L = L_0 + \{a_0\} + L_1 + \cdots + \{ a_{n-1}\} + L_{n}$. Sometimes we say that $L$ has been partitioned into the intervals $L_i$ for $0 \leq i \leq n$.
\end{definition}
\begin{definition}
A linear order $L$ of Hausdorff rank $\alpha$ is \emph{simple} if $L = \sum_{i \in \omega} L_i + \sum_{j\in \omega^*} L_j$ with $\lim_{i \rightarrow \omega} r(L_i)+1 = \alpha$ or $\lim_{j \rightarrow \omega^*} r(L_j)+1 = \alpha$.
If $\sum_{i \in \omega} L_i\cong 0$ ($\sum_{j\in\omega^*} L_j\cong 0$), then we say that $L$ is \emph{simple of type $\omega$} (\emph{simple of type $\omega^*$}).
\end{definition}
Simple linear orders are ``simple'' in the sense that they serve as building blocks of other linear orders of their Hausdorff rank. In particular, every linear order of Hausdorff rank $\alpha$ can be finitely partitioned, so that each of the intervals is simple or of lower Hausdorff rank. Simple linear orders play a central role in the proof of \cref{thm:mainthm}. Their key property is that they have Scott sentences of lesser complexity than usual linear orders of their rank. We will prove the following theorem in the process of proving \cref{thm:mainthm}. 
\begin{theorem}\label{thm:mainthm2}
Let $L$ be a simple linear order with $r(L) = \alpha$. Then $L$ has a $\Pinf{2 \alpha + 1}$ Scott sentence.
\end{theorem}

Before proceeding with the proofs of \cref{thm:mainthm,thm:mainthm2} we observe some easy facts that we will use without referencing.
\begin{proposition}
  Let $M, N$ be linear orders where $r(M)=\alpha$ and $M\leq_{2\alpha+1} N$. Then
  \begin{enumerate}
    \item $r(N)=\alpha$,
    \item if $M$ is simple, then $N$ is simple of the same type.
  \end{enumerate}
\end{proposition}
\begin{proof}
We prove (2) of the proposition by cases, depending on whether $M$ is simple of type $\omega$ or $\omega+\omega^*$. The case when $M$ is simple of type $\omega^*$ is similar to the first case and thus omitted. Item (1) of the proposition is proven along the way.

\emph{Case 1:} Suppose $M = \Sigma_{i \in \omega} M_i$ where each $M_i$ is of rank $\beta_i$. Note that each of the following sentences is $\Pinf{2 \alpha +1}$, so since they are true of $M$ they will also be true of $N$. \\
\begin{align*}
   \text{ (1) } & \forall x \forall y (x \sim_{\alpha} y) \\
   \text{ (2) } & \bigwedge_{n \in \omega} \exists x_0 \ldots x_n ( x_0 \not \sim_{\beta_0} x_1 \not \sim_{\beta_{1}} \cdots \not \sim_{\beta_{n-1}} x_n ) \land (x_0 < \cdots < x_n) \\
   \text{ (3) } & \exists x \forall y (y < x \rightarrow y \sim_{\beta_0} x)
\end{align*}
    Sentence (1) states that there is at most one $\alpha$-block.  If $\alpha$ is a limit ordinal, sentence (2) says that $M$ contains an increasing sequence of structures whose ranks are cofinal in $\alpha$, or that there are infinitely many $\beta$-blocks if $\alpha = \beta +1$. Sentence (3) says that there is a first $\beta_0$-block.
    Observe (1) implies $r(N) \leq \alpha$, and (2) implies $r(N) \geq \alpha$.

\emph{Case 2:} Suppose that $M$ is simple of type $\omega + \omega^*$, where $M = \Sigma_{i \in \omega} M_i + \Sigma_{i \in \omega^*} M_i'$ and each $M_i$, $M_i'$ is of rank $\beta_i$, $\beta_i'$, respectively. Then the following sentences are $\Pi_{2 \alpha + 1}$ and true of $M$:
\begin{align*}
    \text{ (1) } & \forall x \forall y \forall z (x \sim_\alpha y \lor x \sim_{\alpha} z) \\
    \text{ (2) } & \bigwedge_{n \in \omega} \exists x_0 ... x_n \exists x_0' ... x_n' ( x_0 \not \sim_{\beta_0} x_1 \not \sim_{\beta_1} \cdots &\not \sim_{\beta_{n-1}} x_n \not \sim_{\beta_n} x_n' \not \sim_{\beta'_{n-1}} \cdots \not \sim_{\beta_0'} x_0') \\
    &  & \land (x_0 < \cdots < x_n < x_n' < \cdots x_0')\\
    \text{ (3) } & \exists x \forall y (y < x \rightarrow y \sim_{\beta_0} x)\\
    \text{ (4) } & \exists x \forall y (y > x \rightarrow y \sim_{\beta_0'} x)
\end{align*}
\end{proof}

\subsection{Base case}

The proof we give for the base case of \cref{thm:mainthm} is quite different from the ones for successor and limit cases because we want to show something stronger, namely that every linear order of Hausdorff rank $1$ has a computable $d$-$\Sigma_3$ Scott sentence, and if simple or of order type $m+\zeta+n$ where $m,n\geq 0$, has a computable $\Pi_3$ Scott sentence. In \cref{sec:tightness} we show that these Scott sentences are optimal. 
\begin{theorem}\label{thm:hr1dsicom3}
    Let $L$ be of Hausdorff rank $1$. Then it has a $\dSicom{3}$ Scott sentence.
\end{theorem}
\begin{proof}
We begin by defining some useful auxiliary formulas. First, note that the successor relation
$$S(x,y) \text{ iff } x< y \land \forall z \left(x\leq z\leq y \ra x=z\lor z=y\right)$$
is $\Picom{1}.$ Next, define $$\phi^r(x)= \forall y \left( \left(y \geq x \land y \sim_1 x \right) \ra \exists z \ S(y,z) \right).$$
This $\Picom{3}$ formula says that there are infinitely many elements to the right of $x$ but in the same $1$-block. Changing $\geq$ to $\leq$ in the above definition and using the predecessor relation instead of the successor relation, one can similarly define $\phi^l$. Finally, define

\begin{multline*} \phi^m(x)= \forall x_0 \ldots x_{m+1}\left( \bigwedge_{0\leq i\leq m+1} x_i\sim_1 x \ra \left(\bigvee_{0\leq i<j\leq m+1} x_i=x_j\right)\right)
    \\
    \land \exists x_1 \ldots x_m \bigwedge_{0\leq i\leq m} x_i \sim_1 x \land x_1<x_2\cdots < x_m.
    \end{multline*}
This formula is $d$-$\Sicom{2}$ and says that $x$ lies in a finite $1$-block of size $m$.

Write $L=L_1+\dots +L_n$ in $1$-block form. Let

 \begin{equation*}
    \phi_i(x) =
    \begin{cases*}
      \phi^r(x)                           & if $L_i \cong \omega$ \\
      \phi^l(x)                           & if $L_i \cong \omega^*$ \\
      \phi^r(x) \land \phi^l(x)           & if $L_i \cong \zeta$ \\
      x = x                               & if $L_i$ is finite
    \end{cases*}
  \end{equation*}

 and

  \begin{equation*}
    \psi_i(x) =
    \begin{cases*}
      \neg \phi^l(x)                          & if $L_i \cong \omega$ \\
      \neg \phi^r(x)                          & if $L_i \cong \omega^*$ \\
       x = x                                  & if $L_i \cong \zeta$ \\
      \phi^m(x)                               & if $L_i$ is finite of size $m$.
    \end{cases*}
  \end{equation*}

Note that each $\phi_i(x)$ is $\Picom{3}$ and $\psi_i(x)$ is $\Sicom{3}$.
Let $\phi_{Ax}$ denote the conjunction of the axioms of linear orders. Then the following is a $d$-$\Sigma_3^{\text{c}}$ Scott sentence for $L$:
    \begin{multline*} \exists x_1,\dots, x_n\left(\bigwedge_{1\leq i <j\leq n}\left( x_i\not \sim_1 x_j \land x_i< x_j \right) \land \bigwedge_{1\leq i\leq n} \psi_i(x_i)\right)\\
    \land \forall x_1,\dots,x_n \left(\bigwedge_{1\leq i<j\leq n}\left( x_i\not \sim_1 x_j \land x_i<x_j \right)\ra \bigwedge_{1\leq i\leq n} \phi_i(x_i)\right)\\
    \land \forall x_0,\dots, x_{n+1} \left( \bigvee_{0\leq i<j\leq n+1} x_i \sim_1 x_j\right) \land
    \phi_{Ax} .
    \end{multline*}

\end{proof}
The following is easy to obtain from what we have established so far. We thus state it without proof.
\begin{proposition}\label{prop:simplehr1pi3}
    If $r(L) = 1$ and $L$ is simple or of order type $m+\zeta+n$ with $m,n\geq\omega$, then it has a $\Picom{3}$ Scott sentence.
\end{proposition}

\subsection{Inductive Step}
For the base case, we gave the Scott sentence explicitly in order to conclude that the sentence was also computable. We now use combinatorial notions such as $\alpha$-freeness to establish the general case. 




\begin{theorem}\label{thm:simple1}
Let $M$ be a simple linear order with $r(M) = \alpha$. Suppose that for any scattered linear orders $M',N'$ of Hausdorff rank $\beta<\alpha$, $M' \leq_{2 \beta + 2} N' \Rightarrow M' \cong N'$.
Then $M \leq_{2 \alpha + 1} N \Rightarrow M \cong N$.
\end{theorem}
\begin{proof}
We will only consider the case where $L$ is simple of type $\omega$; the other cases are similar.

\emph{Successor Case:} Suppose $\alpha = \beta + 1$ and write $M = \Sigma_{i \in \omega} M_i$, $N = \Sigma_{i \in \omega} N_i$ in $\beta$-block form.
Now we show that each $M_i \cong N_i$. If not, then some $M_i \not \cong N_i$. Let $i$ be the first such $i$, and pick $b \in N_{i+1}$. Then there exists $a \in M$ such that $(-\infty,a) \cong (-\infty,b)$ since
    $(-\infty,b) \leq_{2 \beta + 2} (-\infty,a)
   \Rightarrow  (-\infty,b) \cong (-\infty,a)$.

Now consider $(-\infty,a)$ and $(-\infty,b)$ in $\beta$-block form.
Since $(-\infty,a)$ and $(-\infty,b)$ have the same number of blocks by the previous lemma, $(-\infty,a)$ must have $i+1$ blocks. It follows that $M_i$ is the $i$-th block of $(-\infty,a)$, and $N_i$ the $i$-th block of $N$. Since $(-\infty,a) \cong (-\infty,b)$, it is impossible that $M_i \not \cong N_i$, since they are corresponding blocks in the block forms of $(-\infty,a)$, $(-\infty,b)$ respectively.

\emph{Limit case}: Suppose $\alpha$ is a limit ordinal.
Write $M = \Sigma_{i \in \omega} M_i$, $N = \Sigma_{i \in \omega} N_i$, where $r(M_i) = \alpha_i, r(N_i) = \beta_i$, and where cofinally many of the $\beta_i's$ are successor ordinals. We may choose strictly increasing subsequences $(\alpha_{i(k)})_{k \in \omega}$, $(\beta_{i(k)} )_{k \in \omega}$ such that
\[\alpha_{i(k)} \leq \beta_{i(k)} < \alpha_{i(\kappa+1)}\]
for each $k$, where each $\beta_{i(k)}$ is a successor ordinal.
\\
Now $M,N$ each have a first $\beta_{i(0)}$-block: call them $M_0', N_0'$ respectively. Then let $T_0^M$ and $T_0^N$ such that
\[M = M_0' + T_0^M\quad \text{ and } \quad N = N_0' + T_0^N.\]
Assume we have defined $T_k^M$, $T_k^N$ by induction; then they each have a first $\beta_{i(k+1)}-$block. Let those blocks be $M_{k+1}'$, $N_{k+1}'$ and define $T_{k+1}^M$, $T_{k+1}^N$ so that
\[M = \Sigma_{j = 0}^{k+1} M_j' + T_{k+1}^M \quad \text{ and } \quad N = \Sigma_{j = 0}^{k+1} N_j' + T_{k+1}^N.\]
Since $sup(\{ \beta_{i(k)} \}_{k \in \omega}) = \alpha$, we have that
\[ M = \Sigma_{k \in \omega} M_k' \quad \text{ and } \quad
N = \Sigma_{k \in \omega} N_k'.\]
where each $M_k', N_k'$ has Hausdorff rank $\beta_{i(k)}$.

If $M \not \cong N$, there is a first $k$ such that $M_k' \not \cong N_k'$. Let $k$ be such. Writing $M,N$ in $\beta_{i(k)}$-block form, the initial segments $\Sigma_{i \leq k} M_i'$, $\Sigma_{i \leq k} N_i'$ of $M,N$ are the first $\beta_{i(k)}$-blocks of $M,N$ respectively.

Since $M \leq_{2 \alpha + 1} N$, for any $b$ in the second $\beta_{i(k)}$-block of $N$ there is an $a \in M$ such that
\begin{align*}(-\infty,b) &\leq_{2 \alpha} (-\infty,a)\\
\Rightarrow (-\infty,b) &\leq_{2\beta_{i(k)} +2} (-\infty,a)\\
\Rightarrow (-\infty,b) &\cong (-\infty,a).\end{align*}
Writing $(-\infty,b)$, $(-\infty,a)$ in $\beta_{i(k)}$-block form, we find that $M_k'$,$N_k'$ are the first $\beta_{i(k)}$-blocks of each, so they must be isomorphic - a contradiction.
\end{proof}

It is easy to see that every scattered linear order can be partitioned into a sum of simple linear orders of the same Hausdorff rank. For example consider a finite partition $(a_0,\dots, a_{n-1})$ of the order $\zeta+\zeta$. Then there must be some interval $[a_i,a_{i+1}]$ of order type $\omega+\omega^*$. This together with the following  lemma is very useful for the analysis of back and forth relations of linear orders. 
\begin{lemma}[{\cite[Lemma 15.7]{ash2000}}]\label{lem:akpart}
Suppose $\bar{a}, \bar{b}$ are finite partitions of $L, L'$ of into the intervals $L_i, L_i'$ for $0 \leq i \leq n$. Then $(L,\ol a)\leq_{\alpha} (L',\ol b)$ if and only if for all $0\leq i\leq n$, $L_i\leq_{\alpha} L_i'$.
\end{lemma}

\begin{theorem}\label{thm:maxwithparameters}
Let $M$ be a scattered linear order of Hausdorff rank $\alpha$. Then there is a tuple $\bar{c}$ of elements in $M$ such that for any $N$ and any tuple of elements $\bar{b}$ from $N$, $(M, \bar{c}) \leq_{2 \alpha + 1} (N, \bar{b}) \Rightarrow (M, \bar{c})  \cong (N, \bar{b})$. Moreover, this $\bar{c}$ can be chosen so that the same holds for any refinement of $\bar{c}$; in other words,
\[ \forall \bar{a} \in M  \ \forall N \ \forall \bar{b}, \bar{b}' \in N \ (M, \bar{c} \bar{a}) \leq_{2 \alpha + 1} (N, \bar{b} \bar{b}') \Rightarrow (M, \bar{c}\bar{a})  \cong (N, \bar{b} \bar{b}').\]
As a corollary, $M \leq_{2 \alpha + 2} N \Rightarrow M \cong N$.
\end{theorem}

\begin{proof}
We proceed by induction on $\alpha$. The base case follows from \cref{thm:hr1dsicom3}. Assume the inductive hypothesis. Then \cref{thm:simple1} holds.  Pick a finite partition $\bar{c}$ of $M$ into a sum of simple linear orders. Then by \cref{lem:akpart}, the result holds. Again by \cref{lem:akpart} and the inductive hypothesis, the result holds for any refinement $\bar{c}\bar{a}$ of the finite partition $\bar{c}$.

It remains to prove the corollary. If $M \leq_{2 \alpha + 2} N$ then for any $\bar{b} \in N$, there is a tuple $\bar{a} \in M$ such that $(N, \bar{b}) \leq_{2 \beta + 1}  (M, \bar{a}).  $ In particular, this will be true for the $\bar{c} \in N$ such that $(N,\bar{c}) \leq_{2 \alpha + 1} (M, \bar{a} ) \Rightarrow (N, \bar{c}) \cong (M, \bar{a})$. So $(N, \bar{c}) \leq_{2 \alpha + 1} (M, \bar{a})$ for some $\bar{a}$, and therefore $M \cong N$.
\end{proof}

If there were only countably many scattered linear orders of Hasudorff rank $\alpha$, then by what we have shown above we would be able to conclude that every rank $\alpha$ linear order has a $d$-$\Sigma_{2 \alpha + 1}$ Scott sentence. However, this is not the case. For example, there are uncountably many linear orders just of Hausdorff rank 2. With a little more work, this is not a stumbling block.

\begin{theorem}\label{thm:simpleSS}
Let $M$ be a simple linear order with $r(M) = \alpha$. Then $M$ has a $\Pi_{2 \alpha + 1}$-Scott sentence.
\end{theorem}

\begin{proof}
It is enough to show that no tuple of $M$ is $2 \alpha$-free.
Note that no tuple in $M$ is $2\alpha$-free if and only if
$\forall \bar{a} \in M, \exists \beta < 2\alpha$ and $\bar{a}' \in M $ such that $ \forall \bar{b}, \bar{b}'$
$ \bar{a} \bar{a}' \leq_{\beta} \bar{b} \bar{b}' \Rightarrow  \bar{a}\leq_{2 \alpha} \bar{b}.$

We will consider only the case where $M = \Sigma_{i \in \omega} M_i$ is simple of type $\omega$; the other cases proceed similarly. 
Fix an arbitrary tuple $\bar{a} \in M$. Let $M_k$ be the largest $k$ such that some element of $\bar{a}$ is in $M_k$. Consider $M'$ an initial segment of $M$ containing $M_k$. Let $r(M') = \gamma$ and without loss of generality assume $M'$ has a last element, say $c$ . By \cref{thm:maxwithparameters} there is a tuple $\bar{a}'$ in $M'$ such that for all $N$ and $\bar{b},d, \bar{b}'\in N$, if $(M', \bar{a}c\bar{a}') \leq_{2 \gamma + 1} (N, \bar{b}d \bar{b}')$ then $(M', \bar{a}c\bar{a}') \cong (N, \bar{b}d \bar{b}')$.

Choose this $\bar{a}'$ and suppose that $\bar{b},c, \bar{b}'$ are any tuples from $M$ such that $\bar{a}c \bar{a}' \leq_{2 \gamma + 1} \bar{b}d \bar{b}'$. Since $(M, \bar{a}c \bar{a}') \leq_{2 \gamma + 1} (M, \bar{b}d \bar{b}')$ it follows from \cref{lem:akpart} that $(M', \bar{a}c \bar{a}') \leq_{2 \gamma + 1} (N', \bar{b} d\bar{b}')$ for some $N'$ an initial segment of $M$. Therefore, $M'\cong N'$, $\sum_{i=0}^k M_i\subseteq M'$ and $\sum_{i=0}^k M_i \subseteq N'$. This implies that the isomorphism between $M'$ and $N'$ is an automorphism on $\sum_{i=0}^k M_i$ and therefore $\ol a \in aut_{\sum_{i=0}^k M_i}(\bar b)$. This clearly implies that $\bar a\in aut_M(\bar b)$.
\end{proof}

\begin{corollary}\label{cor:sigma2alpha2ss}
If $M$ is a scattered linear order with $r(M) = \alpha$, then $M$ has a $\Sigma_{2 \alpha + 2}$ Scott sentence.
\end{corollary}
\begin{proof}
 Let $\bar{c}$ be a finite partition of $M$ into intervals each of which is simple. Then by \cref{lem:akpart}, the same proof as in \cref{thm:simpleSS} will also hold over $\bar{c}$.
\end{proof}

\begin{lemma}\label{lem:ohsoimportantlemma}
Suppose that $L$ is a linear order of Hausdorff rank strictly greater than $\beta$, and that $L'$ is a $\beta$-block of $L$. For increasing $\bar{a}, \bar{b} \in L'$,
\[ (L, \bar{a}) \leq_{2 \beta + 1} (L, \bar{b}) \Rightarrow (L', \bar{a}) \leq_{2 \beta + 1} (L', \bar{b}).\]
\end{lemma}
\begin{proof}
For the base case, assume $r(L) \geq 1$ and $\ol a, \ol b \in L'$ are in the same $0$-block. Then $|\ol a| = |\ol b|=1 $ and $\ol a=\ol b$.

Now assume the lemma is true for all $\gamma<\beta$ and that $(L,\ol a)\leq_{2\beta+1} (L,\ol b)$. Without loss of generality, suppose $|\ol a|=|\ol b|=n-1$. Then $\ol{a}, \ol{b}$ are distinct finite partitions of $L$, say into the intervals $L_{a_i}$ and $L_{b_i}$ respectively. Define $L'_{a_i}$, $L'_{b_i}$ likewise. It is sufficient to show that $L'_{a_i}\leq_{2\beta+1} L'_{b_i}$.
For $0<i<n$ this follows trivially as for these $i$, $L'_{a_i}=L_{a_i}$ and $L'_{b_i}=L_{b_i}$. We need to show that $L'_{a_0}\leq_{2\beta+1} L'_{b_0}$ and that $L'_{a_n} \leq_{2\beta+1}L'_{b_n}$.

We will show $L'_{a_0}\leq_{2\beta+1} L'_{b_0}$; the case that $L'_{a_n}\leq_{2\beta+1} L'_{b_n}$ is similar. To this end, pick ordered $\ol c \in L'_{b_0}$. Then $(c_0,b_0)$ has Hausdorff rank $\gamma < \beta$. Since $(L,a_0) \leq_{2 \beta + 1} (L,b_0)$, there is $\ol d \in L$ such that $(L,\ol cb_0) \leq_{2 \beta} (L, \ol da_0)$. Now as $(c_0,b_0)$ is of Hausdorff rank $\gamma$ and $2\gamma+1< 2\gamma+2 \leq 2\beta$ we have that $r((d_0,a_0))=r((c_0,b_0))=\gamma$ and thus $\ol d\in L'_{a_0}$. Then by the inductive hypothesis we get that for all $\ol c \in L_{b_0}'$ there is $\ol d \in L_{a_0}'$ such that $(L'_{b_0},\ol c)\leq_{2\beta}( L'_{a_0},\ol d)$ and hence $L'_{a_0}\leq_{2\beta+1} L'_{b_0}$.
\end{proof}
To complete the proof of \cref{thm:mainthm}, by \cref{thm:rachael} it remains to show that all Hausdorff rank $\alpha$ linear orders have a $\Pi_{2\alpha+2}$ Scott sentence. 
\begin{theorem}\label{thm:pi2alpha2ss}
Every scattered linear order $L$ with $r(L)$ = $\alpha$ has a $\Pi_{2 \alpha + 2}$ Scott sentence.
\end{theorem}
\begin{proof}

We proceed by induction on $\alpha$ to show that no tuple in $L$ is $(2\alpha+1)$-free. The base case has already been shown, so suppose $\alpha > 1$. Since  $r(L) = \alpha$ we may write $L$ as $L_0 + \cdots + L_n$ where each $L_i$ is a distinct $\alpha$-block. We will first show that no $\ol a \in L_i^{<\omega}$ for any $i \leq n$ is $(2\alpha+1)$-free and then reason that no tuple can be $(2\alpha+1)$-free.

Without loss of generality let $\ol a\in L_i^{<\omega}$ be increasing, $\ol a = (a_0, \ldots, a_m)$, and assume towards a contradiction that it is $(2\alpha+1)$-free in $L$. 
Then
\[ \forall (\beta < 2\alpha+1) \forall \ol b \exists \ol a',\ol b'\ \left(\ol a\ol b \leq_\beta \ol a'\ol b' \land \ol a\not \leq_{2\alpha+1} \ol a'\right).\]
We distinguish cases depending on whether $L_i$ has a first or last element. We first deal with the case when it has neither. 

Choose $\ol b$ so that each component is from a distinct $L_j$ with $j\neq i$. Then by $(2\alpha+1)$-freeness there are $\ol a',\ol b'$ such that $\ol a\ol b \leq_{2\alpha} \ol a'\ol b'$ but $\ol a\not \leq_{2\alpha+1} \ol a'$. Recall that the relation $\sim_\alpha$ is $\Sigma_{2\alpha}$-definable. Since for all $j,k < n$ with $j\neq k$, $b_j\not \sim_\alpha b_k$, it follows that the same is true for $\ol b'$. Moreover, for all $j \leq m$ and all $k< n$, $a_j\not \sim_\alpha b_k$ so $a_j'\not \sim_\alpha b_k'$. Since there are only $n+1$ $\alpha$-blocks in L, $\ol a'\in L_i$. 
Pick an arbitrary $c > a_{m}'$ in $L_i$. Then $(a_{m}',c)$ has Hausdorff rank $\gamma<\alpha$ and thus a $\dSinf{2\gamma+1}$ Scott sentence by the inductive hypothesis. Let $\phi(x,y)$ be this Scott sentence relativized to the interval $(x,y)$. Then
\[ (L,\ol a'\ol b')\models \exists y\ \phi(a_{m}',y)\] 

Similarly, pick an arbitrary $c' < a_0'$ in $L_i$ and let $\phi'(x,y)$ be the relativized Scott sentence of $(c', a_0')$. Then
\[ (L,\ol a'\ol b')\models \exists x\ \phi(x, a_0) \land \exists y\ \phi'(a_m,y).\]
Denote the above formula by $\psi(\ol x)$. It is at most $\Sigma_{2\alpha}$, so
\[ (L,\ol a\ol b)\models \psi(\ol a).\]
Furthermore, $\ol a\leq_{2\alpha} \ol a'$ implies that $(a_0,a_k)\cong (a'_0,a'_k)$. We just proved the following.
\begin{claim}\label{claim:intervalsiso}
For every interval $(c',d')\subset L_i$ such that $\ol a'\in (c',d')$ there is an interval $(c,d)\subset L_i$ and an isomorphism $f:(c',d')\cong (c,d)$ such that $f(\ol a')=\ol a$.
\end{claim}
Now, for every interval $(c,d)\subset L_i$ including $\ol a$ consider the set
\[ I_{(c,d)}=\{ (\ol a\ol e,\ol a'f(\ol e)): \ol e\in (c,d), f: (c,d) \cong (c',d'), f(\ol a)=\ol a'\}\]
and let
\[ I=\bigcup_{(c,d)\subset L_i, \ol a\in (c,d)} I_{(c,d)}.\]
We need to show that $I$ has the back and forth property. By \cref{claim:intervalsiso} it is non-empty and all the pairs are partial isomorphisms. We now show that for every $(\ol a\ol e,\ol a'\ol e')\in I$ and $b\in L_i$ there is $b'\in L_i$ such that $(\ol a\ol e b,\ol a'\ol e' b')\in I$. Let $c,d$ such that $(\ol a\ol e, \ol a'\ol e')\in I_{(c,d)}$. If $b\in (c,d)$, then this clearly holds. Assume $b>d$, then $b$ and the greatest element of $\ol a\ol e$ are at most finitely many $\beta$ blocks apart for some $\beta<\alpha$, say they are $k$ $\beta$-blocks apart. Pick an element $y'$ $k+1$ $\beta$-blocks from the biggest element in $\ol a'\ol e'$. Then, by \cref{claim:intervalsiso}, there is an element $y$ such that $(c,y)\cong(c',y')$ and the interval $(c,y)$ contains $(c,d)$ as an initial segment.
Let $g:(c,d)\cong (c',d')$ such that $\ol a\ol e\mapsto \ol a'\ol e'$. Then $g$ is an isomorphism between initial segments of $(c,y)$ and $(c',y')$ and it can be extended to an isomorphism $g': (c,y) \cong (c',y')$. To see this, just notice that $[d,y)\cong [d',y')$, say by $h$. Then for $x\in (c,y)$, let $g'(x)=g(x)$, and for $x\in [d,y)$ let $g'(x)=h(x)$. Thus $(\ol a\ol e b,\ol a'\ol e' g'(b))\in I_{(c,y)}$ and therefore also in $I$. The case when $b<c$ is symmetric. It remains to show that for every $(\ol a\ol e,\ol a'\ol e')\in I$ and $b'\in L_i$ there is $b\in L_i$ such that $(\ol a\ol e b,\ol a'\ol e' b')\in I$.  Let $c,d$ such that $(\ol a\ol e,\ol a'\ol e')\in I_{(c,d)}$ and $c',d'$ be the elements such that $(c,d)\cong (c',d')$. We assume that $b'>d'$, the other cases again being similar or trivial. Take any element $y'\in L_i$ such that $y'>b'$. Then $b'\in (c',y')$. By \cref{claim:intervalsiso} there is $y$ such that $(c',y')\cong (c,y)$ and hence some $b\in(c,y)$ such that $(\ol a\ol eb,\ol a\ol eb')\in I_{(c,y)}$. It follows that $I$ has the back and forth property. We have that $(\ol a,\ol a')\in I$ and thus $\ol a'\in aut_{L}(\ol a)$ and $\ol a \leq_{2\alpha+1} \ol a'$, a contradiction to our assumption that $\ol a$ is $(2\alpha+1)$-free.

If $L_i$ has a first element, then we proceed similarly as above but this time we choose $\ol b$ of length $n+1$ with $b_j\in L_j$ for $j\leq n$ and such that $b_i$ is the first element of $L_i$. Notice that if $L_i$ has a first element then there is a $\beta<\alpha$ such that the $\beta$-block form of $L_i$ is well-ordered. Then, since $\ol a\ol b\leq_{2\alpha} \ol a'\ol b'$, we have that $[b_i,a_m)\cong [b_i',a_m')$. We now proceed as above but with the difference that we take \[I=\bigcup_{[a_m,d)\subseteq L_i}\{ (a_m \ol e, f(a_m)f(\ol e)): \ol e\in [a_m,d), f:[a_m,d)\cong [a_m',d'), f(a_m)=a_m'\}.\]
The case when $L_i$ has a last element is symmetric and if $L_i$ has both a last and a first element then it has Hausdorff rank $\beta<\alpha$ and thus $\ol a$ can not be $(2\alpha+1)$-free by hypothesis.

We still need to show that no tuple containing elements from different $\alpha$-blocks can be $(2\alpha+1)$-free. The following claim settles this.
\begin{claim}
 Let $\ol a\in L^{<\omega}$ where without loss of generality $\ol a=\ol a_0\dots \ol a_n$ with $\ol a_i\in L_i$ for $i<n$. Then $\ol a$ is not $(2\alpha+1)$-free in $L$ if and only if for all $i<n$ $\ol a_i$ is not $(2\alpha+1)$-free in $L$.
\end{claim}
\begin{proof}
We proof the case where $\ol a$ consists of two elements $a_0<a_1$ in different blocks. The general case follows easily.

Assume $a_0a_1$ is $(2\alpha+1)$-free, then we have that there is $a_0'a_1'$ in the same blocks as $a_0$ and $a_1$ respectively with
\[ a_0a_1\leq_{2\alpha}a_0'a_1' \land a_0a_1\not \leq_{2\alpha+1} a_0'a_1'.\]
(We get this by choosing the right $\ol b$ as always.)
But then $a_0\leq_{2\alpha} a_0'$ and $a_1\leq_{2\alpha} a_1'$ and we have seen above that thus $a_0 \leq_{2\alpha+1} a_0'$ and $a_1\leq_{2\alpha+1} a_1'$. Then $L=L_0+a_0+L_1+a_1+L_2$ and $L=L_0'+a_0'+L_1'+a_1'+L_2'$.
By \cref{lem:akpart} and $a_i\leq_{2\alpha+1} a_i'$ for $i<2$ we get that $L_j\leq_{2\alpha+1}L_j'$ for $j\in \{0,2\}$. It remains to show that $L_1\leq_{2\alpha+1} L_1$. We have that $a_0+ L_1+a_1=a_0+\hat L + L^\alpha + \tilde L+a_1$ where $\hat L$ and $\tilde L$ are end, respectively initial segments of the blocks of $a_0$ and $a_1$ and $L^\alpha$ are the blocks in between; we also get the same for $a_0'$ and $a_1'$. Then $L^\alpha\cong L'^{\alpha}$ as $a_0 \sim_\alpha a_0'$ and $a_1\sim_\alpha a_1'$. But we have that $\tilde L \cong \tilde L'$ and $\hat L \cong \hat L'$ since those are parts of the blocks of $a_0,a_1$ respectively and $a_i'\in aut(a_i)$ for $i<2$. Thus $L_1\cong L_1'$, and again by \cref{lem:akpart} we thus have that $a_0a_1\leq_{2\alpha+1} a_0'a_1'$, a contradiction.
\end{proof}
\end{proof}
Let $L$ be a linear order with $r(L)=\alpha$. Then by \cref{cor:sigma2alpha2ss} it has a $\Sigma_{2\alpha+2}$ Scott sentence and by \cref{thm:pi2alpha2ss} it has a $\Pi_{2\alpha+2}$ Scott sentence. Thus, by \cref{thm:rachael} we have that it has a $\dSinf{2\alpha+1}$ Scott sentence. This proves \cref{thm:mainthm}.

As a corollary we obtain a new proof of a theorem by Nadel~\cite{nadel_scott_1974} that shows that no scattered linear order has high Scott rank, i.e., the Scott rank of every scattered $L$ is less than $\omega_1^L$.
\begin{corollary}
 Let $L$ be a linear order and $r(L)=\alpha$ for a countable ordinal $\alpha$. Then $L$ has categoricity Scott rank $\beta<\omega_1^{L}$.
\end{corollary}
\begin{proof}
    Nadel~\cite{nadel_scott_1974} showed that any scattered linear order $L$ has Hausdorff rank less than $\omega_1^L$.
    For a modern proof of this one can consider the relativization of a result of Montalb\'an~\cite{montalban2005} which says that every hyperarithmetic linear order has Hausdorff rank less than $\omega_1^{\mrm CK}$.
    Clearly, every linear order $L$ is hyperarithmetic in itself and thus $r(L)<\omega_1^{\mrm L}$. By \cref{thm:mainthm} $L$ has a $\dSinf{2\alpha+1}$ Scott sentence and hence it has categoricity Scott rank $2r(L)+1$. Clearly, if $r(L)$ is $L$-computable then so is $2r(L)+1$ and thus $L$ has categoricity Scott rank less than $\omega_1^{L}$.
\end{proof}
\section{Optimality and Index sets}
\label{sec:tightness}
In this section we prove several results about index sets and the optimality of the above results.
We first show that the bounds obtained in \cref{sec:upperbound} are tight.
\begin{proposition}
  For each countable ordinal $\alpha$ there are linear orders $L_1$, $L_2$ such that
  \begin{enumerate}
    \item $L_1$ has a $\Pinf{2\alpha+1}$ Scott sentence but no $\dSinf{2\alpha}$ Scott sentence.
    \item $L_2$ has a $\dSinf{2\alpha+1}$ Scott sentence but no $\Pinf{2\alpha+1}$ or $\Sinf{2\alpha+1}$ Scott sentence,
  \end{enumerate}
\end{proposition}
\begin{proof}
  Ash~\cite{ash1986} gave a complete characterization of the back and forth relations for ordinals, see also~\cite[Lemma 15.10]{ash2000}. We will use his characterization to obtain the required examples.

  \emph{ad (1).} By~\cite[Lemma 15.10]{ash2000} $\omega^{\alpha+1}\leq_{2\alpha} \omega^{\alpha}$ and $\omega^{\alpha}\leq_{2\alpha} \omega^{\alpha+1}$. Thus, $\omega^{\alpha}$ cannot have a $\dSinf{2\alpha}$ Scott sentence as otherwise $\omega^{\alpha+1}\cong \omega^{\alpha}$, a contradiction.

  \emph{ad (2).} Again by~\cite[Lemma 15.10]{ash2000} we have that $\omega^{\alpha}\cdot 3 \leq_{2\alpha+1} \omega^{\alpha}\cdot 2 \leq_{2\alpha+1} \omega^{\alpha}$. Thus, if $\omega^{\alpha}\cdot 2$ had a $\Pinf{2\alpha+1}$ or a $\Sinf{2\alpha+1}$ Scott sentence, then $\omega^{\alpha}\cong \omega^{\alpha}\cdot 2$, respectively, $\omega^{\alpha}\cdot 3\cong \omega^{\alpha}\cdot 2$. This is clearly a contradiction.
\end{proof}
We now show that the bound on the complexity of the Scott sentences of linear orders of Hausdorff rank $1$ calculated in \cref{thm:hr1dsicom3} is optimal.
\begin{proposition}\label{prop:indexsetomega}
  Let $L$ be isomorphic to $\omega$, $\omega^*$, or $\zeta$. Then its index set is $\Pi_3^0$ complete.
\end{proposition}
\begin{proof}
We only give the proof for $\omega$, the proofs for $\omega^*$ and $\zeta$ follow the same scheme. Let $(\C_i)_{i\in \omega}$ be a computable enumeration of partial structures and assume that $P\subseteq \omega$ is $\Pi_3^0$. We will build a computable function $f$ such that
\[ \C_{f(p)}\cong \omega \LR p\in P. \]
Given $P$, there is a computable function $g$ such that
\[ p\in P \LR \forall x\ W_{g(p,x)} \text{ is finite.}\]
We build the structure $\C_{f(p)}$ in stages, i.e., $\C_{f(p)}=\lim_{s} \C_{f(p),s}$. Let  $\C_{f(p),0}\cong\omega$. Assume we have defined $\C_{f(p),s}$; at stage $s+1$ for every $x$ and every $y$ such that $y\enters W_{g(p,x),s}$ add a fresh element between any two elements in the interval $[x, x+1]$.\footnote{We work under the standard assumption that only $y<s$ may enter $W_{g(p,x),s}$.} This finishes the construction.

\emph{Verification:} Assume $p\in P$; then for all $x$, $W_{g(p,x)}$ is finite and thus by construction the intervals $[x,x+1]$ are all finite. Hence, $\C_{f(p)} \cong \omega$. If $p\not\in P$, then there is $x$ such that $W_{g(p,x)}$ is infinite and thus between any two elements in $[x,x+1]$ we can find another element. Therefore $(x,x+1)$ is dense and $\C_{f(p)}\not \cong \omega$.
\end{proof}
\begin{corollary}\label{cor:indexsethr1}
  Let $L$ be a linear order such that $r(L)=1$, then $I_L$ is $\Pi_3^0$ hard.
\end{corollary}
\begin{proof}
  Let $L$ be a linear order with $r(L)=1$. We have that $L$ is a finite sum of blocks of type $\omega$, $\omega^*$, $\zeta$ or $n$ and that it contains at least one $\omega$, $\omega^*$, or $\zeta$ block. Assume without loss of generality that  $L$ is computable and that $L=L_1+L_2+L_3$ where $L_2$ is isomorphic to either $\omega$, $\omega^*$, or $\zeta$, and $L_1$, $L_2$, $L_3$ are all computable and disjoint.
  Fix a $\Pi^0_3$ set $P$ and let $(\C_i)_{i\in \omega}$ be a computable enumeration of partial structures. Consider the computable function $g$ such that
  \[ \mc C_{g(p)}=L_1+\mathcal{C}_{f(p)}+L_3\]
  where $\mathcal{C}_{f(p)}$ is the structure constructed in the proof of \cref{prop:indexsetomega} but with the same universe as $L_2$. Then clearly
  \[ \mc C_{g(p)}\cong L \LR p\in P\]
  and thus $I_L$ is $\Pi_3^0$ hard.
\end{proof}
\cref{prop:simplehr1pi3} and the relativizations of \cref{prop:indexsetomega} and \cref{cor:indexsethr1} show the following.
\begin{theorem}
  Simple linear orders of Hausdorff rank 1 have $\Picom{3}$ optimal Scott sentences.
\end{theorem}

\begin{lemma}\label{lem:notsimplehr1}
Let $L$ be a linear order such that $r(L)=1$. Then $L$ is not simple or of order type $m+\zeta+n$ with $m,n\geq 0$ if and only if $L$ contains an interval $I$ isomorphic to $\omega+\omega$, $\omega^*+\omega^*$, $\omega+\zeta$, $\zeta+\omega^*$, or $\omega+n+\omega^*$ for some $n>0$.
\end{lemma}
\begin{proof}
The direction from right to left follows directly from the definition of simple linear orders. For the other direction assume that $L$ is not simple or of order type $m+\zeta+n$. Then the order type of $L/{\sim}$ must be greater or equal to $2$. If $L/{\sim}\cong 2$, then again by the definition of simple it must contain an interval isomorphic to one of the above orders. If $L/{\sim}\cong n$ with $n>2$, then let $L'$ be the restriction of $L$ to its first $3$ blocks.

Assume that $L'=L_1+L_2+L_3$ is the block form of $L'$. We distinguish three cases:
\begin{enumerate}
\item If $rk(L_1)=rk(L_2)=1$, then either $L_1+L_2$ contains an interval isomorphic to $\omega+\omega$, $\omega^*+\omega^*$, $\omega+\zeta$, or $\zeta+\omega^*$ and we are finished. Or $L_1+L_2\cong \omega+\omega^*$. In that case, notice that $rk(L_3)=1$ and that $L_3$ does not have a left limit. It follows that $L'$ contains an interval isomorphic to $\omega^*+\omega^*$. The case when $rk(L_2)=rk(L_3)=1$ is symmetric.
\item If $rk(L_2)=0$, then $L_2$ must contain an element which is a left limit in $L'$ and an element which is a right limit in $L'$. Thus, $L'$ contains an interval isomorphic to $\omega+n+\omega^*$ for some $n$. 
\item If $L'\cong m+\zeta+n$ for some $n,m>0$. Then let $L_4$ be the fourth block of $L$. The block $L_3=n$ has a right limit in $L$ and thus $L$ contains an interval isomorphic to $\omega+n+\omega^*$.

\end{enumerate}
\end{proof}

\begin{lemma}\label{lem:sigma3hard2blocks}
Let $L$ be a linear order such that $r(L)=1$ and $L$ contains an interval isomorphic to $\omega+\omega$, $\omega^*+\omega^*$, $\omega+\zeta$, or $\zeta+\omega$. Then $I_L$ is $\Sigma^0_3$ hard.
\end{lemma}
\begin{proof}
We may assume without loss of generality that $L=L_1+L_2+L_3$ where $L_2$ is of order type $\omega+\omega$, $\omega^*+\omega^*$, $\omega+\zeta$, or $\zeta+\omega$, and all $L_i$ are computable. We first deal with the case that $L_2\cong \omega+\omega$. The constructions for the other cases are pretty similar. We will show how to adapt our construction to these cases at the end of the proof.

Consider any $\Sigma^0_3$ set $S$. We build a computable function $f$ such that
\[ \mc C_{f(e)}\cong \omega +\omega \LR e\in S.\]
As $S$ is $\Sigma^0_3$ there is a computable function $g$ such that
\[ e\in S \LR \exists x\ W_{g(e,x)} \text{ is infinite.}\]
We may furthermore assume without loss of generality that if $e\in S$ then there exists unique $x$ such that $W_{g(e,x)}$ is infinite, see~\cite[Theorem 4.3.11]{soare2016}. We build the structure $\C_{f(e)}$ in stages. The elements $\langle x,0\rangle $ will be the potential limit points of the first copy of $\omega$ we are building.

\emph{Construction:} The structure $\C_{f(e),0}$ has universe $\{ \langle x,0\rangle : x\in \omega\}$ and $\langle x,0\rangle <\langle y,0\rangle$ if and only if $x<y$ for all $x,y$.
Assume we have defined $\mc C_{f(e),s}$ and are at stage $s+1$ of the construction. For every $x<s$ check if there is $y\in W_{g(e,x),s}\setminus W_{g(e,x),s-1}$ and if so add $\langle x,s+1\rangle$ to the end of $[\langle x,0\rangle,\langle x+1,0\rangle)$. This finishes the construction.

\emph{Verification:} Clearly $\mc C_{f(e)}=\lim_s \mc C_{f(e),s}$ is computable. If $e\in S$, then there is exactly one $x$ such that $W_{g(e,x)}$ is infinite. Hence $\langle x+1,0\rangle$ is a limit point, $(\infty,\langle x+1,0\rangle)\cong \omega$, $[\langle x+1,0\rangle,\infty)\cong \omega$ and thus $C_{f(e)}\cong \omega+\omega$. On the other hand, if $e\not\in S$, then for all $x$, $W_{g(e,x)}$ is finite. Therefore $\mc C_{f(e)}$ will not contain a limit point and will be isomorphic to $\omega$.

Without loss of generality $L_2$ is computable and thus we may pull back $\mc C_{f(e)}$ to have the same universe. We can now define $\mc C_{g(e)}$ such that
\[ \mc C_{g(e)}\cong L_1+ \mc C_{f(e)}+L_3\cong L \LR e\in S.\]
This shows that $I_L$ is $\Sigma^0_3$ hard.

If $L$ contains an interval of type $\omega^*+\omega^*$, use $\mc (C_{f(e)})^*$ instead of $\mc C_{f(e)}$. If $L$ contains an interval of type $\omega+\zeta$ let $\mc C_{h(e),0}$ be a copy of $\omega$ as above. At stage $s+1$, if there is $y\in W_{g(e,x),s}\setminus W_{g(e,x),s-1}$, then add $\langle x,2s+2\rangle$ to the end of the interval $[\langle x,0\rangle, \langle x,2t+2\rangle]$ where $t$ is the last stage less than $s$ where $W_{g(e,x),{t}}\setminus W_{g(e,x),t-1}\neq \emptyset$ and add $\langle x,2s+3\rangle$ to the beginning of the interval $[\langle x,2t+3\rangle, \langle x+1,0\rangle]$. It is not hard to see that then $\mc C_{h(e)}$ is as required. For $\zeta+\omega^*$ use $\mc (C_{h(e)})^*$.
\end{proof}
\begin{lemma}
 Let $L$ be a linear order that contains an interval isomorphic to $\omega+n+\omega^*$ for some $n>0$. Then $I_L$ is $\Sigma^0_3$ hard.
\end{lemma}
\begin{proof}
Similarly to the proof of \cref{lem:sigma3hard2blocks} we may assume that $L=L_1+L_2+L_3$ where $L_2\cong \omega+n+\omega^*$. Let $S$ be a $\Sigma^0_3$ set and $g$ be a computable function such that
\[ e\in S \LR \exists x\ W_{g(e,x)}\text{ is infinite}.\]
We will build a structure $\mathcal C_{f(e)}$ in stages such that
\[ C_{f(e)}\cong \omega+n+\omega^* \LR e\in S.\]
Our construction is essentially a priority construction where for each $x$ we have a worker trying to build an $n$ block marked by constants $c^x_1,\dots, c^x_n$. If $x<y$ and $x$ acts at some stage $s$, then it will \emph{initiate} all $y>x$, i.e., it will reset the constants $c^y_1,\dots, c^y_n$ to be undefined. If $c^x_1,\dots, c^x_n$ is undefined at stage $s$ then we say that \emph{$x$ is initiated}.

\emph{Construction:} The structure $\mathcal{C}_{f(e),0}$ is empty and all $x$ are initiated. Assume we have defined the structure $\mathcal C_{f(e),s}$ and are at stage $s+1$ of our construction. Let $x$ be the least such that $W_{g(e,x),s}\setminus W_{g(e,x),s-1}\neq \emptyset$ and proceed as follows:
\begin{enumerate}
    \item If $x$ is initated and all other $y$ are initiated, then no worker has acted before and thus $\mathcal C_{f(e),s}$ is empty. Let $c^x_i=\langle s,i\rangle$, and let $\C_{f(e),s+1}$ be the elements $c^x_i$ ordered lexicographically.
    \item If $x$ is initiated and there is $y$ which is not initiated, then let $y_0$ be the least such $y$. Let $c^x_i=\langle s,i\rangle$, and let the universe of $\C_{f(e),s+1}^\star$ be the universe of $\C_{f(e),s}$ union $\{\langle s,i\rangle: 0<i\leq n\}$. Order $\C_{f(e),s+1}^\star$ by extending $\C_{f(e),s}$ such that the $c^x_i$ are lexicographically ordered and $c^x_1$ is the successor of $c^{y_0}_n$. 
    \item Obtain $\C_{f(e),s+1}$ by adding $\langle s,0\rangle$ to the end of the interval $(-\infty, c^x_1)$ and $\langle s,n\rangle$ to the start of the interval $(c^x_n,\infty)$.
    \item Initiate all $y>x$.
\end{enumerate}

\emph{Verification:}
Assume $e\in S$. Let $x_0$ be the least such that $W_{g(e,x_0)}$ is infinite. Then $x_0$ is allowed to act infinitely many times and there exists a stage $t$ such that for all $s>t$ $x_0$ is not initiated. Notice that if $y>x_0$ is acting at stage $s>t$, then there is a stage $r>s$ at which $y$ is initiated again. Hence, $(c^{x_0}_n,\infty)\cong \omega^*$ and $(-\infty,c^{x_0}_1)\cong \omega$ and thus $\C_{f(e)}\cong \omega+n+\omega^*$.

On the other hand, if $e\not\in S$, then no $x$ acts infinitely often. Therefore, for no $x$, the element $c^x_n$ will be a left limit a the end of the construction and thus $\C_{f(e)}$ will not be isomorphic to $\omega+n+\omega^*$.
\end{proof}
\begin{corollary}
Let $L$ be a linear order with $r(L)=1$ that is not simple or of order type $\zeta$. Then it has a $\dSinf{3}$ optimal Scott sentence.
\end{corollary}
We cannot hope to obtain that our upper bounds are optimal for all linear orders of Hausdorff rank greater than $1$. Frolov and Zubkov (unpublished) gave for all $n<\omega$ and $3\leq m\leq 2n$ examples of linear orders of Hausdorff rank $n$ that have degree of categoricity $\Delta^0_m$. This implies that these orders are $\Delta^0_m$ categorical. An analysis of their proofs shows that they are even uniformly $\Delta^0_m$ categorical. Hence, by \cref{thm:ssnopar} these orders have Scott sentences of strictly less complexity.
They kindly allowed us to print one of their examples.
\begin{proposition}[after Frolov and Zubkov]
There is a linear order of Hausdorff rank $2$ that has a $\Picom{4}$ Scott sentence.
\end{proposition}
\begin{proof}
Consider the following linear order which clearly has Hausdorff rank $2$.
\[ L=\sum_{i\in \omega} i+ \zeta\]
We will show that it is uniformly $\Delta^0_3$ categorical. Consider a copy $L'$ of $L$. We define a $\Delta^{L\oplus L'}_3$ computable isomorphism $f:L\ra L'$.

Assume we have defined $f(y)$ for $y<x$ and define $f(x)$ as follows. Let $n=1$.
\begin{enumerate}
    \item If $x$ is in a block of order type $n$ proceed to 2; else proceed to 3.
    \item Say $x$ is the $i^{th}$ element in its block. Find a block of order type $n$ in $L'$ and set $f(x)$ to be the $i^{th}$ element in this block.
    \item If there is a block of order type $n$ with all elements smaller than $x$ and if there is a block of order type $n+1$ with all elements bigger than $x$ proceed to 4; else proceed to 5.
    \item Locate the $\zeta$ block between $n$ and $n+1$ in $L'$ and define $f(x)$ so that in the limit $f$ will be an isomorphism between the $\zeta$ blocks.
    \item Increase $n$ by $1$ and go to step 1.
\end{enumerate}
It remains to show that $f$ is $\Delta^{L\oplus L'}_3$ computable. A block of size $n$ is definable in $L$ by
\begin{multline*} \exists x_1,\dots ,x_n \forall x( x<x_1 \rightarrow  \exists y\  x<y<x_1) \land \forall x( x>x_n \ra \exists y\ x_n<y<x)\\\land x_1<\dots <x_n \land \forall y_1,\dots, y_{n+1} (\bigwedge_{i<n+1} x_1\leq y_1 \leq x_n\ra \bigvee_{1\leq i< j\leq n+1} y_i=y_j)\end{multline*}
As $L$ and $L'$ both contain exactly one block of size $n$, $\Delta^{L\oplus L'}_3$ can find it and thus all of the steps in the construction are computable in $\Delta^{L\oplus L'}_3$. If $x$ is in a block of size $i$ then the procedure will terminate at step $2$ after $i$ iterations, and if $x$ is in a $\zeta$ block that succeeds a block of size $i$ it will terminate at step $4$ after $i$ iterations. Notice that defining the isomorphism between elements of that $\zeta$ block can be done computably in $\Delta_2^{L\oplus L'}$. Thus the whole procedure can be done computably in $\Delta_3^{L\oplus L'}$. Furthermore it is uniform in $L$ and $L'$ and thus the $L$ is uniformly $\Delta^0_3$ categorical. Using an effective version of \cref{thm:ssnopar} we get that $L$ has a $\Picom{4}$ Scott sentence. More formally, we use a result by R.\ Miller~\cite[Proposition 4.1]{miller2017a} to obtain a c.e. Scott family of $\Sigma^{L\oplus L'}_3$ formulas and the relativization of a result by Alvir, Knight, and McCoy~\cite[Proposition 2.9]{alvir2018} to get the required Scott sentence.
\end{proof}
In order to obtain results on structures of arbitrary Hausdorff rank we use a version of Ash and Knight's pairs of structure theorem which appeared in~\cite{ash1990}. Before we state the precise theorem we need some more definitions.
\begin{definition}
 Let $\A$ be a structure and $K$ be a class of structures, all in the same language. For countable $\alpha$, define $K\leq_\alpha \A$ if for all $\beta<\alpha$, and every $\ol a\in A^{<\omega}$ there exists $\B\in K$ and $\ol b \in B^{<\omega}$ such that $(\A,\ol a)\leq_\beta (\B,\ol b)$.
\end{definition}
\begin{definition}
 A finite or countable sequence $\langle \A_0,\A_1,\dots \rangle$ of structures is \emph{$\alpha$-friendly} if the structures $\A_i$ are uniformly computable and for $\beta<\alpha$ the back and forth relations $\leq_\beta$ on the set of pairs $(\A_i, \ol a)$ for $\ol a\in A_i^{<\omega}$ are c.e., uniformly in $\beta$.
\end{definition}

The following lemma is restated from~\cite[Theorem 4.2]{ash1990}.
\begin{lemma}\label{thm:pairsofstructureclasses}
    Let $\alpha$ be a computable ordinal, $\A$ be a structure and $K$ be a countable class of structures, all in the same language, such that $K\leq_\alpha \A$ and $K\cup \{\A\}$ is $\alpha$-friendly. Then for each $\Pi^0_\alpha$ set $S$ there is a uniform computable sequence of structures $(\mc C_n)_{n\in\omega}$ such that 
    \[ \mc C_n \cong \begin{cases}
       \A & \text{if }n\in S,\\
       \B & \text{for some }\B\in K, \text{if } n\not\in S.
    \end{cases}\]
\end{lemma}

The following example is obtained by applying \cref{thm:pairsofstructureclasses}. It first appeared in~\cite{ash1990}.
\begin{example}\label{ex:AshKnightEx}
Let $\alpha$ be a computable ordinal, and $S$ be a $\Sigma_{2\alpha}^0$ set. Then there exists a uniformly computable sequence of structures $(\mathcal C_n)_{n\in\omega}$ such that
\[ \mathcal C_n \cong
\begin{cases}
   \gamma<\omega^{\alpha} &\text{if } n\in S,\\
   \omega^{\alpha} & \text{otherwise}.
\end{cases}
\]
\end{example}
We can now prove the following hardness result about the block relation.
\begin{proposition}\label{thm:completeblock}
For any computable ordinal $\alpha$ and $\Sigma_{2\alpha}^0$ set $S$ there is a linear order $L$ and a computable function $f:\omega \ra L^{2}$ such that
\[ n\in S \LR f(n) \in \sim_{\alpha}.\]
\end{proposition}
\begin{proof}
Let $(\mathcal{C}_n)_{n\in\omega}$ be the sequence of structures from \cref{ex:AshKnightEx}. Let $c_n$ be the first element of $\mathcal C_n$ and assume without loss of generality that the sequence $(c_n)_{n\in\omega}$ is uniformly computable. Let $f: n\mapsto (c_n,c_{n+1})$ and $L=\sum_{n\in \omega} \mc C_n$. It is not hard to see that $c_n\sim_{\alpha}c_{n+1}$ if and only if $n\in S$.
\end{proof}
\cref{thm:pairsofstructureclasses} also allows us to produce the following modification of \cref{ex:AshKnightEx} which we will use to show the completeness of the index set of linear orders with a fixed Hausdorff rank.

\begin{example}\label{ex:AKmodified}
Let $\alpha$ be a computable ordinal, and $S$ be a $\Sigma_{2\alpha+2}^0$ set. Then there exists a uniformly computable sequence of structures $(\mathcal C_n)_{n\in\omega}$ such that
\[ \mathcal C_n \cong
\begin{cases}
   \omega^{\alpha}\cdot m &\text{for some }m\in\omega, \text{if } n\in S,\\
   \omega^{\alpha+1} & \text{otherwise}.
\end{cases}
\]
\end{example}
\begin{proof}
We want to apply \cref{thm:pairsofstructureclasses} with $K=\{\omega^\alpha\cdot m: m\in\omega\}$ and $\A=\omega^{\alpha+1}$. That we can choose $K$ and $\A$ to be $(2\alpha+2)$-friendly follows from \cite[Proposition 15.11]{ash2000}. It remains to show that $K\leq_{2\alpha+2} \omega^{\alpha+1}$. By definition $K\leq_{2\alpha+2} \omega^{\alpha+1}$ if and only if there exists $m\in\omega$ such that for all $\ol a \in \omega^{\alpha+1}$, there exists $\ol b\in \omega^{\alpha}\cdot m$ and $(\omega^{\alpha+1},\ol a)\leq_{2\alpha+1} (\omega^{\alpha}\cdot m,\ol b)$. Clearly, every $\ol a$ of size $n$ splits $\omega^{\alpha+1}$ into intervals 
\[ \A_0 + a_0+ \A_1 + a_1 + \cdots + a_{n-1}+ \omega^{\alpha+1}\]
where each $\A_i<\omega^{\alpha+1}$. Picking $m$ large enough we can find a tuple $\ol b$ of length $n$ in $\omega^{\alpha}\cdot m$ such that $\ol b$ splits $\omega^{\alpha}\cdot m$ into intervals
\[ \B_0+ b_0 + \B_1 + b_1 + \cdots + b_{n-1}+ \omega^{\alpha}\]
where each $\B_i\cong \A_i$. By \cref{lem:akpart} it suffices to show that $\omega^{\alpha+1}\leq_{2\alpha+1} \omega^{\alpha}$. This follows from Ash's characterization of the back and forth relations of ordinals~\cite{ash1986}.
\end{proof}
\begin{remark}
    We argued that $\omega^{\alpha+1}\leq_{2\alpha+1} \omega^{\alpha}$ using Ash's characterisation of the back and forth relations of ordinals~\cite{ash1986}. For this result, a common citation, which we also use frequently in this article, is the restatement in the book by Ash and Knight~\cite[Lemma 15.10]{ash2000}. However, in this restatement there is a typo which would not allow our last argument to go through. Using transfinite induction one can easily check that the version that appeared in~\cite{ash1986} is correct.
\end{remark}
Using the sequence provided by \cref{ex:AKmodified} we immediately get completeness of the index set of linear orders of a given Hausdorff rank.
\begin{theorem}\label{thm:completehr}
For every computable ordinal $\alpha$ the set $\{ \mc C_e : r(\mc C_e)=\alpha\}$ is $\Sigma_{2\alpha+2}^0$ complete.
\end{theorem}

\subsection{Relativizing index sets}
\cref{thm:completeblock} shows that there is no computable infinitary sentence in the language of linear orders of complexity less than $\Sigma_{2\alpha}$ saying that two elements are in the same $\alpha$-block and \cref{thm:completehr} says that the index set of Hausdorff rank $\alpha$ linear orders is $\Sigma_{2\alpha+2}^0$ complete for computable $\alpha$. We want to show that it is not possible to find less complex non-computable sentences defining these properties for any countable $\alpha$.

For $\alpha\geq \omega_1^{\mrm CK}$, $\omega^{\alpha}$ does not have a computable copy and therefore is not $(2\alpha)$-friendly. However, take $X$ such that $X$ computes a notation for $\omega^{\alpha}$, i.e., $\omega^{\alpha}<\omega_1^{X}$. Then we can canonically relativize being $(2\alpha)$-friendly to $X$ and, because the proof of~\cite[Theorem 15.11]{ash2000} also relativizes, get that $\omega^{\alpha}$ is $(2\alpha)$-friendly relative to $X$. Then, \cref{thm:pairsofstructureclasses} relativizes in the following sense.
\begin{lemma}
   Let $\alpha$ be an $X$-computable ordinal for some $X\subseteq \omega$, $\A$ be a structure and $K$ be a class of structures, all in the same language, such that $K\leq_\alpha \A$ and $K\cup \{\A\}$ is $\alpha$-friendly relative to $X$. Then for each $\Pi^0_\alpha(X)$ set $S$ there is a uniformly $X$-computable sequence of structures $(\mc C_n)_{n\in\omega}$ such that 
    \[ \mc C_n \cong \begin{cases}
       \A & \text{if }n\in S,\\
       \B & \text{for some }\B\in K, \text{if } n\not\in S.
    \end{cases}\]
\end{lemma}
It is a well known fact that the function indexing the sequence $(\mc C_n)_{n\in\omega}$ can be chosen computably and this together with the above mentioned property of the required ordinals is sufficient to obtain the analogues of \cref{ex:AshKnightEx,ex:AKmodified} for ordinals computable relative to some set $X$. 
Fixing $\alpha<\omega_1$ and relativizing the examples to any set $X$ we obtain the desired results.
\begin{corollary} Let $\alpha$ be a countable ordinal.
\begin{enumerate}
\item The relation $\sim_\alpha$ is not definable by a $\mc L_{\omega_1\omega}$ formula less complex than $\Sigma_{2\alpha}$ in the language of linear orders.
\item\label{item:hralphaoptimal} The class of Hausdorff rank $\alpha$ linear orders is not axiomatizable by a formula of complexity less than $\Sigma_{2\alpha+2}$.
\end{enumerate}
\end{corollary}
Note that by Vaught's theorem \cref{item:hralphaoptimal} is equivalent to: The class of Hausdorff rank $\alpha$ linear orders is $\pmb{\Sigma}_{2\alpha+2}$-complete.

\printbibliography

\end{document}